\UseRawInputEncoding
\documentclass[leqno]{amsart}

\usepackage{stmaryrd,graphicx}
\usepackage{amssymb,mathrsfs,amsmath,amscd,amsthm,color}
\usepackage{float}
\usepackage{mathabx}
\usepackage[all,cmtip]{xy}
\DeclareMathAlphabet{\mathpzc}{OT1}{pzc}{m}{it}
\usepackage{amsfonts,latexsym,wasysym}
\usepackage{cleveref}

\newcommand{\down}{\downarrow}
\newcommand{\up}{\uparrow}

\newcommand{\sw}{\wt{\bigvee}}

\newcommand{\ui}{I}

\newcommand{\wt}{\widetilde}

\newcommand{\mci}{\mathcal{I}}
\newcommand{\mcj}{\mathcal{J}}

\newcommand{\scra}{\mathscr{A}}
\newcommand{\scrb}{\mathscr{B}}
\newcommand{\scrc}{\mathscr{C}}
\newcommand{\scrd}{\mathscr{D}}
\newcommand{\scre}{\mathscr{E}}
\newcommand{\scrf}{\mathscr{F}}

\newcommand{\scrp}{\mathscr{P}}
\newcommand{\scru}{\mathscr{U}}

\newcommand{\scrs}{\mathscr{S}}
\newcommand{\scrv}{\mathscr{V}}
\newcommand{\bba}{\mathbb{A}}

\newcommand{\bbe}{\mathbb{E}}

\newcommand{\bbh}{\mathbb{H}}
\newcommand{\bbn}{\mathbb{N}}

\newcommand{\bbz}{\mathbb{Z}}

\renewcommand{\dh}{\mathbf{dh}}
\newcommand{\uh}{\mathbf{uh}}

\newcommand{\ov}{\overline}

\newcommand{\x}{\overline x_0}
\newcommand{\im}{\text{Im}}
\newtheorem{theorem}{Theorem}[section]
\newtheorem{lemma}[theorem]{Lemma}
\newtheorem{proposition}[theorem]{Proposition}
\newtheorem{corollary}[theorem]{Corollary}
\theoremstyle{definition}\newtheorem{definition}[theorem]{Definition}
\newtheorem{example}[theorem]{Example}

\newtheorem{remark}[theorem]{Remark}

\begin{document}
\title[Fundamental groups of reduced suspensions]{Fundamental groups of reduced suspensions are locally free}

\author[J. Brazas]{Jeremy Brazas}
\address{West Chester University\\ Department of Mathematics\\
West Chester, PA 19383, USA}
\email{jbrazas@wcupa.edu}

\author[P. Gillespie]{Patrick Gillespie}
\address{University of Tennessee\\ Department of Mathematics\\
Knoxville, TN 37996, USA}
\email{pgilles5@vols.utk.edu}

\date{\today}

\begin{abstract}
In this paper, we analyze the fundamental group $\pi_1(\Sigma X,\x)$ of the reduced suspension $\Sigma X$ where $(X,x_0)$ is an arbitrary based Hausdorff space. We show that $\pi_1(\Sigma X,\x)$ is canonically isomorphic to a direct limit $\varinjlim_{A\in\scrp}\pi_1(\Sigma A,\x)$ where each group $\pi_1(\Sigma A,\x)$ is isomorphic to a finitely generated free group or the infinite earring group. A direct consequence of this characterization is that $\pi_1(\Sigma X,\x)$ is locally free for any Hausdorff space $X$. Additionally, we show that $\Sigma X$ is simply connected if and only if $X$ is sequentially $0$-connected at $x_0$.
\end{abstract}

\maketitle

\section{Introduction}

When a based space $(X,x_0)$ is well-pointed, i.e. the inclusion $\{x_0\}\to X$ is a cofibration, the reduced suspension $\Sigma X$ with canonical basepoint $\x$ is homotopy equivalent to the unreduced suspension $SX$ and it follows that $\pi_1(\Sigma X,\x)$ is free on the set of path components of $X$ not containing $x_0$. When $X$ is not well-pointed, the situation becomes more delicate but includes some of the most important examples in the algebraic topology of locally complicated spaces. For example:
\begin{enumerate}
\item If $\bbe_0=\{1,2,3,\dots\}\cup\{\infty\}$ is the one-point compactification of the natural numbers, then $\Sigma \bbe_0$ is homeomorphic to the usual ($1$-dimensional) infinite earring space $\bbe_1$ whose fundamental group is extensively studied and applied \cite{CChe,Edafreesigmaproducts}. Moreover, the iterated suspension $\Sigma^n\bbe_0$ is the $n$-dimensional earring \cite{EK00higher} sometimes called the Barratt-Milnor Spheres \cite{BarrattMilnor}.
\item If $X=\{(0,0)\}\cup\{(x,\sin(1/x))\mid 0<x<1\}$ is the topologist's sine curve, then $\Sigma X$ is the harmonic archipelago $\bbh\bba$, introduced in \cite{BS98}, and whose fundamental group has also been widely studied, c.f. \cite{CHMArchipelago,Corsontriplecone,Corsonpmods,HHarchip,Hojkauniversal}.
\item If $X=C\bbe_0=\bbe_0\times \ui/\bbe_0\times\{1\}$ is the unreduced cone over a convergent sequence (a converging fan of arcs) and the image of $(\infty,0)$ is taken to be the basepoint of $X$, this provides what is perhaps the most well-known example where $X$ is a path-connected compact metric space but $\pi_1(\Sigma X)$ is non-trivial (in fact, it is isomorphic to the archipelago group $\pi_1(\bbh\bba)$)
\end{enumerate}
The archipelago group $\pi_1(\bbh\bba)$ is readily shown to be a direct limit of sequence of earring groups \cite[Prop. 1]{Hojkauniversal}. Hence, it is natural to ask if fundamental groups of reduced suspensions can always be ``built out of" free groups and earring groups. In this paper, we answer this question in the affirmative by establishing methods that allow us to characterize the algebraic structure of $\pi_1(\Sigma X,\x)$ in substantial generality. Independently, the authors of \cite{CorsonHojkaRed} have studied fundamental groups of reduced suspensions, restricting their attention to $X$ which are first countable at $x_0$ and focusing on the establishment of a dichotomy that allows one to decide when the groups $\pi_1(\Sigma X,\x)$ contain a copy of the archipelago group $\pi_1(\bbh\bba)$. We find little overlap with their results and methods.

For our main result, we consider a based space $(X,x_0)$ and let $\scrp$ denote the set of all subsets $A\subseteq X$ containing $x_0$ for which (1) $A$ has countably many path components, (2) each path component of $A$ is a Peano continuum, and (3) the path components of $A$ cluster at $x_0$. When ordered by inclusion, $\scrp$ becomes a directed set (see \Cref{directed}). The main technical achievement of this paper is the following.

\begin{theorem}\label{mainthm}
Let $(X,x_0)$ be a based Hausdorff space. The inclusion maps $A\to X$, $A\in\scrp$ induce a canonical isomorphism $\psi:\varinjlim_{A\in\scrp} \pi_1(\Sigma A, \x)\to \pi_1(\Sigma X, \x)$.
\end{theorem}

We also show that each of the groups $\pi_1(\Sigma A, \x)$ is either isomorphic to a finitely generated free group or the infinite earring group $\pi_1(\bbe_1)$ (\Cref{loc-free}) and is therefore locally free. Since direct limits of locally free groups are locally free \cite[Lemma 24]{CHMArchipelago}, the following consequence is immediate.

\begin{theorem}\label{homology}
For every based Hausdorff space $(X,x_0)$, $\pi_1(\Sigma X,\x)$ is locally free.
\end{theorem}

It is certainly possible to extend the results of \cite{CorsonHojkaRed} using \Cref{mainthm}; however, we avoid doing so here. Rather, we briefly mention two other consequences of our methods. In \Cref{equiv}, we give a complete characterization of the spaces $X$ for which $\Sigma X$ is simply connected. Additionally, since abelianization is a left adjoint functor, it preserves direct limits. Thus $\psi$ descends to an isomorphism $\psi':\varinjlim_{A\in\scrp} H_1(\Sigma A)\to H_1(\Sigma X)$ on singular homology. It is known that $H_1(\bbe_1)$ is isomorphic to $\bbz^{\bbn}\oplus (\bbz^{\bbn}/\oplus_{\bbn}\bbz)$ \cite{EKH1ofHE} and $H_1(\bbh\bba)$ is isomorphic to $\bbz^{\bbn}/\oplus_{\bbn}\bbz$ \cite{HHcotorsion,KarimovRepovs} but that such isomorphisms are highly non-constructive.

\begin{corollary}
For any based Hausdorff space $(X,x_0)$, $H_1(\Sigma X)$ is canonically isomorphic to a direct limit of groups each of which is (not canonically) isomorphic to either $\bbz^n$ or $\bbz^{\bbn}\oplus (\bbz^{\bbn}/\oplus_{\bbn}\bbz)$. In particular, $H_1(\Sigma X)$ is torsion-free.
\end{corollary} 

\section{Notation and Preliminaries}

Throughout, $X$ will denote a Hausdorff topological space with basepoint $x_0$. The closed unit interval will be denoted $I$ and $I^2$ is the closed unit square. A \textit{path} in $X$ is a continuous map $\alpha:I\to X$, which we call a \textit{loop} if $\alpha(0)=\alpha(1)$. We let $\alpha\cdot\beta$ denote the usual \textit{concatenation} of paths when $\alpha(1)=\beta(0)$ and we let $\alpha^{-}(t)=\alpha(1-t)$ denote the \textit{reverse} of $\alpha$. The constant path at $x\in X$ is denoted by $c_x$.

If $[a,b],[c,d]\subseteq \ui$ and $\alpha:[a,b]\to X$, $\beta:[c,d]\to X$ are maps, we write $\alpha\equiv\beta$ if $\alpha=\beta\circ h$ for some increasing homeomorphism $h: [a,b]\to [c,d]$; if $h$ is linear and if it does not create confusion, we will identify $\alpha$ and $\beta$. Note that $\equiv$ is an equivalence relation. We write $\alpha\simeq\beta$ if $\alpha$ and $\beta$ are homotopic relative to their endpoints. Certainly, if $\alpha\equiv\beta$, then $\alpha\simeq\beta$ by a homotopy with image in $\im(\alpha)$.

\begin{definition}
Given $x_0\in X$, we say that a countable collection $\scra$ of subsets of $X$ \textit{clusters at} $x_0$ if for every neighborhood $U$ of $x_0$, we have $A\subseteq U$ for all but finitely many $A\in\scra$. If $\scra$ is indexed by $\bbn=\{1,2,3,\dots\}$, we may say that $\scra$ \textit{converges toward} $x_0$.

As special cases, we also define the following. If $C\subseteq X$ is a countable subset, we say $C$ \textit{clusters at} $x_0$ if $\scra=\{\{c\}\mid c\in C\}$ clusters at $x_0$. Similarly, if $\scrf=\{f_j\mid j\in J\}$ is a countable collection of maps $f_j:Y_j\to X$, we say $\scrf$ \textit{clusters at} $x_0$ if $\{\im(f_j)\mid j\in J\}$ clusters at $x_0$.
\end{definition}

Let $(A_j,a_j)$, $j\in J$ be a countable collection of based spaces. The \textit{shrinking wedge} of this collection is the space $\sw_{j\in J}(A_j,a_j)$ whose underlying set is the usual one-point union $\bigvee_{j\in J}(A_j,a_j)$ with canonical basepoint $a_0$ but where $U\subseteq \sw_{j\in J}(A_j,a_j)$ is open if and only if $U\cap A_j$ is open in $A_j$ for all $n\in\bbn$ and if $A_j\subseteq U$ for all but finitely many $j\in J$ whenever $a_0\in U$. Note that $\sw_{j\in J}(A_j,a_j)$ may be canonically embedded as a closed subspace of the direct product $\prod_{j\in J}A_j$. Moreover, by restricting to individual wedge-summands, we see that based maps $f:(\sw_{j\in J}(A_j,a_j),a_0)\to (X,x_0)$ are in one-to-one correspondence with collections $\{f_j\}_{j\in J}$ of based maps $f_j:A_j\to X$ that cluster at $x_0$.

Let $\mci$ be a countable linearly ordered set and $\{\alpha_i\}_{i\in\mci}$ be collection of loops $\alpha_i:(\ui,\partial\ui)\to (X,x_0)$ that clusters at $x_0$. The $\mci$-indexed concatenation of this collection is the loop $\prod_{i\in\mci}\alpha_i$ defined as follows: let $\scru$ be a collection of disjoint open intervals in $(0,1)$ such that (1) $\bigcup\scru$ is dense in $[0,1]$ and (2) when $\scru$ is equipped with the natural ordering inherited from $(0,1)$ there is an order isomorphism $\psi:\mci\to\scru$. Then $\prod_{i\in\mci}\alpha_i$ maps $\ui\backslash\bigcup\scru$ to $x_0$ and agrees with $\alpha_i$ on $\ov{\phi(i)}$. Up to the equivalence relation $\equiv$, the definition of $\prod_{i\in\mci}\alpha_i$ does not depend on the choice of $\scru$.

\subsection{Reduced Suspensions}

For based space $(X,x_0)$, the quotient space \[\Sigma X=\frac{X\times [0,1]}{\{x_0\}\times [0,1]\cup X\times \{0,1\}}\] is the \textit{reduced suspension} of $(X,x_0)$. Let $q:X\times [0,1]\to \Sigma X$ denote the canonical quotient map. The canonical basepoint of $\Sigma X$, which is the image of $\beth=\{x_0\}\times [0,1]\cup X\times \{0,1\}$, will be denoted $\x$. Since $q$ maps $(X\backslash \{x_0\})\times (0,1)$ homeomorphically onto $\Sigma X\backslash \{\x\}$, we will often identify these spaces. Given any point $x\in X$, there is a canonical path $\lambda_x:I\to \Sigma X$ given by $\lambda_x(t)=q(x,t)$ called the \textit{adjoint path at} $x$. We refer to $\lambda_x$ as the \textit{adjoint path at} $x$ since the function $x\mapsto \lambda_x$ is the adjoint of the identity map $\Sigma X\to \Sigma X$. Note that $\lambda_{x_0}$ is the constant path at $\x$.

Let $\bbe_0=\{1,2,3,\dots\}\cup\{\infty\}$ be the one-point compactification of the natural numbers with basepoint $\infty$. This case is of particular interest since convergent sequences $\{x_n\}_{n\in\bbn}\to x_0$ in $X$ correspond uniquely to based maps $(\bbe_0,\infty)\to (X,x_0)$. We refer to $\bbe_0$ as the \textit{$0$-dimensional earring space} since $\bbe_1=\Sigma \bbe_0$ is the \textit{$1$-dimensional earring space}, i.e. $\bbe_1$ is canonically homeomorphic to the planar set $\bigcup_{n\in\bbn}C_n$ where $C_n$ the circle of radius $1/n$ centered at $(1/n,0)$. In this example, $\lambda_{n}:\ui\to \bbe_1$ is the loop parameterizing the $n$-th circle.

\subsection{Standard neighborhoods}

Since we work in significant generality, it is necessary to describe a convenient basis of neighborhoods at $\x$. Given a pointed open cover $\scrv=\{V_x\mid x\in X\}$ of $X$, i.e. where $x\in V_x$ for all $x\in X$, and any (not necessarily continuous) function $\eta:X\to (0,1/3)$, we may define \[O(\scrv,\eta)=(V_{x_0}\times [0,1])\cup \bigcup_{x\in X}V_x\times ([0,\eta(x))\cup(1-\eta(x),1]).\] We refer to a set of the form $O(\scrv,\eta)$ as a \textit{standard neighborhood of} $\beth$ in $X\times \ui$. Hence, $(V\times I)\cup (X\times [0,1/3)\cup (2/3,1])$ is the largest possible standard neighborhood for given neighborhood $V=V_{x_0}$ of $x_0$. A set of the form $O(\scrv,\eta)$ is saturated with respect to $q$ and we refer to its image $q(O(\scrv,\eta))$ as a \textit{standard neighborhood of }$\x$. The standard neighborhoods of $\x$ form a neighborhood base at $\x$ in $\Sigma X$. Hence, if $f_k:Y\to X\times [0,1]$, $k\in\bbn$ is a sequence of maps such that for every standard neighborhood $O(\scrv,\eta)$ of $\beth$, we have $\im(f_k)\subseteq O(\scrv,\eta)$ for all but finitely many $k\in\bbn$, then $\{q\circ f_k\}_{k\in\bbn}$ is a sequence of maps $Y\to\Sigma X$ that converges toward $\x$. Finally, we establish notation for the following open cover of $O(\scrv,\eta)$: let
\begin{enumerate}
\item $O(\scrv,\eta)^{\down}=(V_{x_0}\times [0,1])\cup \bigcup_{x\in X}V_x\times [0,\eta(x))$,
\item $O(\scrv,\eta)^{\up}=(V_{x_0}\times [0,1])\cup \bigcup_{x\in X}V_x\times (1-\eta(x),1]$.
\end{enumerate}

\begin{remark}\label{closed}
Let $Y$ be a closed subset of $X$ containing $x_0$ and $O(\mathscr{W},\zeta)$ be a standard neighborhood in $Y\times \ui$ with $\mathscr{W}=\{W_y\mid y\in Y\}$. We can extend $\zeta:Y\to (0,1/3)$ to $\eta:X\to (0,1/3)$ by defining $\eta(x)=1/4$ for all $x\in X\setminus Y$. If $x\in Y$, find $V_x$ open in $X$ such that $V_x\cap Y=W_x$ and if $x\in X\backslash Y$ let $V_x=X\backslash Y$. If we set $\scrv=\{V_x\mid x\in X\}$, then $O(\mathscr{V},\eta)\cap (Y\times I)=O(\mathscr{W},\zeta)$. It follows that the inclusion $Y\to X$ induces an embedding $\Sigma Y\to \Sigma X$. Hence, we may identify $\Sigma Y$ naturally as a subspace of $\Sigma X$.
\end{remark}

%

\subsection{Downward and upward-sliding homotopies}

\begin{definition}
Let $S\subseteq \Sigma X$ be a subset.
\begin{enumerate}
\item The \textit{downward hull} of $S$ in $\Sigma X$ is $\dh(S)=\{\x\}\cup\bigcup_{(x,t)\in S\backslash\x}q(\{x\}\times [0,t])$.
\item The \textit{upward hull} of $S$ in $\Sigma X$ is $\uh(S)=\{\x\}\cup\bigcup_{(x,t)\in S\backslash\x}q(\{x\}\times [t,1])$.
\end{enumerate}
\end{definition}

\begin{remark}
Notice that if $A\subseteq B\subseteq \Sigma X$, then $\dh(A)\subseteq \dh(B)$ and $\uh(A)\subseteq \uh(B)$. If $O(\scrv,\eta)$ is a standard neighborhood of $\beth$, and $S\subseteq q(O(\scrv,\eta)^{\down})$, then $\dh(S)\subseteq \dh(q(O(\scrv,\eta)^{\down}))= q(O(\scrv,\eta)^{\down})$ and if $S\subseteq q(O(\scrv,\eta)^{\up})$, then $\uh(S)\subseteq \uh(q(O(\scrv,\eta)^{\up}))= q(O(\scrv,\eta)^{\up})$.
\end{remark}

\begin{definition}
Suppose that $Y$ is a space and $f=(f_1,f_2):Y\to X\times \ui$ is a map.
\begin{enumerate}
\item The \textit{downward homotopy} of $f$ is the map $H_{f}^{\down}:Y\times I\to X\times I$ defined by $H^{\down}_{f}(y,t)=(f_1(y), f_2(y)(1-t))$;
\item The \textit{upward homotopy} of $f$ is the map $H_{f}^{\up}:Y\times I\to X\times I$ defined by $H^{\up}_{f}(y,t)=(f_1(y), f_2(y)(1-t)+t)$.
\end{enumerate}
\end{definition}

\begin{remark}
Note that if $\im(f)\subseteq \Sigma X\backslash\{\x\}$, then $\im(q\circ H_{f}^{\down})=\dh(\im(f))$ and $\im(q\circ H_{f}^{\up})=\uh(\im(f))$. Additionally,
\begin{enumerate}
\item If $\im(f)\subseteq O(\scrv,\eta)^{\down}$, then $\im(H^{\down})\subseteq O(\scrv,\eta)^{\down}$,
\item If $\im(f)\subseteq O(\scrv,\eta)^{\up}$, then $\im(H^{\up})\subseteq O(\scrv,\eta)^{\up}$.
\end{enumerate}
\end{remark}

\begin{lemma}\label{hep}
Let $K$ be a simplicial complex and $L\subseteq K$ a subcomplex. There exists a retraction $r:K\times I\to K\times\{1\}\cup L\times I$. Moreover, if $i:K\times\{1\}\cup L\times I\to K\times I$ denotes the inclusion, we may assume that $i\circ r(\sigma \times I)\subseteq \sigma\times I$ for each simplex $\sigma\subseteq K$.
\end{lemma}

The first part of the lemma is simply the statement that $(K,L)$ has the homotopy extension property. The second part of the lemma is an immediate consequence of standard proofs of the first part, see \cite[Prop 0.16]{Hatcher} for example. We explicitly state the second part of the lemma because it will be useful for us to have a precise control of the retraction produced.

\begin{definition}\label{push-up-down}
Let $(K,L)$ be a pair of simplicial complexes with $L\subseteq K$. Let $f=(f_1,f_2):K\to X\times I$ be a map, and let $H_f^{\down}:K\times I\to X\times I$ be the downward homotopy of $f$. Let $r$ be the retraction of $K\times I$ onto $K\times\{1\}\cup L\times I$ from \Cref{hep}. The restriction of $H^{\down}_f\circ r$ to $K\times\{0\}$ defines a map $f^{\down}_L:K\to X\times I$ which we call the \textit{pushdown of $f$ rel. $L$}. Note that $f^{\down}_L|_L=f|_L$ and that $f^{\down}_L$ has image in 
$$X\times\{0\}\cup\Big\{(f_1(a),t)\mid a\in L, 0\leq t\leq f_2(a)\Big\}.$$ 
In a symmetric manner, if $H_{f}^{\up}:Y\times I\to X\times I$ is the upward homotopy of $f$ and $r$ is the retraction of $K\times I$ onto $K\times\{1\}\cup L\times I$ from \Cref{hep}, then the restriction of $H^{\up}_f\circ r$ to $K\times\{0\}$ defines a map $f^{\up}_L:K\to X\times I$ which we call the \textit{pushup of $f$ rel. $L$}.
\end{definition}

If $L$ is clear from context, we may simply write $f^{\down}$ and $f^{\up}$ for the pushdown and pushup of $f$ respectively.

\begin{remark}\label{updownsimplices}
Note that $f$ is homotopic rel. $L$ to $f_{L}^{\down}$ (respectively $f_{L}^{\up}$) by a homotopy with image in $\bigcup_{k\in K}\{k\}\times[0,f_2(k)]$ (respectively, by a homotopy with image in $\bigcup_{k\in K}\{k\}\times[f_2(k),1]$). Moreover, the condition placed on $r$ from \Cref{hep} in the definition of $f_L^{\down}$ (respectively $f_L^{\up}$) ensures that for each simplex $\sigma\subseteq K$, $f^{\down}_L(\sigma)\subseteq \bigcup_{x\in \sigma}\{x\}\times[0,f_2(x)]$ (respectively, $f^{\up}_L(\sigma)\subseteq \bigcup_{x\in \sigma}\{x\}\times[f_2(x), 1]$). In particular, this means that $q\circ f^{\down}_L(\sigma)\subseteq \dh(q\circ f(\sigma))$ and $q\circ f^{\up}_L(\sigma)\subseteq \uh(q\circ f(\sigma))$ for each simplex $\sigma\subseteq K$.
\end{remark}

\begin{example}\label{pathexample}
The first situation in which we will apply pushdowns and pushups is the case where $(K,L)=(I,\partial I)$. If $\alpha=(\alpha_1,\alpha_2):\ui\to X\backslash\{x_0\}\times (0,1)$ is a path from $(a,s)$ to $(b,t)$, then $\alpha\simeq\alpha_{\partial I}^{\down}$ in $X\times \ui$. Moreover, when viewing $\alpha$ as a path $\alpha:\ui\to \Sigma X$, we have
\[\alpha \simeq q\circ \alpha_{\partial I}^{\down}\equiv (\lambda_{a})|_{[0,s]}^{-}\cdot c_{\x}\cdot \lambda_{b}|_{[0,t]}.\]
by a homotopy with image in $\dh(\im(q\circ \alpha))$. Similarly, $\alpha$ is path-homotopic to
\[q\circ \alpha_{\partial I}^{\up}\equiv (\lambda_{a})|_{[s,1]}\cdot c_{\x}\cdot (\lambda_{b}|_{[t,1]})^{-}.\]
by a homotopy with image in $\uh(\im(q\circ \alpha))$.
\end{example}

\section{Deforming arbitrary loops in $\Sigma X$}

\begin{definition}
Let $\scrc(X,x_0)$ denote the set of countable subsets of $X$ containing $x_0$ that cluster at $x_0$. If it will not cause confusion we may write $\scrc$ instead of $\scrc(X,x_0)$.
\end{definition}

If $C\in \scrc$, then, since $X$ is assumed to be Hausdorff, $C$ is closed in $X$ and either finite or homeomorphic to $\bbe_0$. Thus $\Sigma C$ may be viewed as a subspace of $\Sigma X$ (recall \Cref{closed}) and is homeomorphic to either a point, a finite wedge of circles, or $\bbe_1$.

\begin{definition}
We say loop $\alpha:[a,b]\to \Sigma X$ is \textit{irreducible} if $\alpha^{-1}(\x)=\{a,b\}$.
\end{definition}

\begin{lemma}\label{simplelooplemma}
For every irreducible loop $\alpha:\ui\to \Sigma X$, there exists $C\in\scrc$ and a surjective loop $\beta:\ui\to \Sigma C$, which is path-homotopic to $\alpha$ in $\Sigma X$. Moreover, if $\im(\alpha)$ lies in standard neighborhood $q(O(\scrv,\eta))$, then this path-homotopy may be chosen to have image in $q( O(\scrv,\eta))$.
\end{lemma}

\begin{proof}
Let $\alpha:(I,\partial I)\to (\Sigma X, \x)$ be an irreducible loop. Find a strictly increasing function $a:\bbz\to (0,1)$, written $a(n)=a_n$ such that as $n\to \infty$, $a(n)\to 1$ and $a(-n)\to 0$. Set $\alpha_n=\alpha|_{[a_n,a_{n+1}]}$ so that $\alpha\equiv\prod_{n\in\bbz}\alpha_n$. Since $\alpha$ is irreducible, we have $\im(\alpha_n)\subseteq (X\backslash\{x_0\})\times (0,1)$ for all $n\in\bbz$ and thus we may write $\alpha(a_n)=(y_n,s_n)$ for $y_n\in X\backslash\{x_0\}$ and $s_n\in (0,1)$. Consider the open cover of $X\times I$ by sets $U^{\down}=X\times [0,2/3)$ and $U^{\up}=X\times (1/3,1]$. Since $\alpha_n:I\to X\times I$ is continuous, we may find a partition $a_{n}=t_{(n,0)}<t_{(n,1)}<t_{(n,2)}<\cdots <t_{(n,k_n)}=a_{n+1}$ such that $\alpha_n$ maps each interval $[t_{(n,j)},t_{(n,j+1)}]$ into either $U^{\down}$ or $U^{\up}$. Notice that $T=\{t_{(n,k)}\mid n\in\bbz,0\leq j\leq k_n\}$ is a closed subset of $I\backslash \partial I$ with $\inf(T)=0$, $\sup(T)=1$, and which has the order type of $\bbz$. Choose an order isomorphism $a':\bbz\to T$. By replacing $a$ with $a'$, we may assume from the start that $\alpha_n$ has image in either $U^{\down}$ or $U^{\up}$ for all $n\in\bbz$.

Recall the pushup and pushdown construction in \Cref{push-up-down}, particularly the case described in \Cref{pathexample}. If $\im(\alpha_n)\subseteq U^{\down}$, we have homotopy $H_{\alpha_n}^{\down}:[a_n,a_{n+1}]\times\ui\to X\times \ui$ and the pushdown $\alpha_{n}^{\down}:[a_n,a_{n+1}]\to X\times \ui$ rel. $\partial I$. Then $q\circ H_{\alpha_n}^{\down}$ can be modified to a path-homotopy $G_n:[a_n,a_{n+1}]\times\ui\to \Sigma X$ from $\alpha_n$ to $\gamma_n=q\circ \alpha_{n}^{\down}$ with image in $\dh(\im(\alpha_n))$. On the other hand, if $\im(\alpha_n)\nsubseteq U^{\down}$, we have $\im(\alpha_n)\subseteq U^{\up}$ and we consider the homotopy $H_{\alpha_n}^{\up}:[a_n,a_{n+1}]\times\ui\to X\times I$ and the pushup $\alpha_{n}^{\up}:[a_n,a_{n+1}]\to X\times \ui$ rel. $\partial I$. Then $q\circ H_{\alpha_n}^{\up}$ can be modified to a path-homotopy $G_n:[a_n,a_{n+1}]\times\ui\to \Sigma X$ from $\alpha_n$ to $\gamma_n=q\circ \alpha_{n}^{\up}$ with image in $\uh(\im(\alpha_n))$.

Let $G:\ui^2\to \Sigma X$ be defined so that $G(\{0,1\}\times\ui)=\x$ and to agree with $G_n$ on $[a_n,a_{n+1}]\times\ui$. To verify the continuity of $G$, it suffices to check the continuity of $G$ at points in $\{0,1\}\times\ui$. Let $O(\scrv,\eta)$ be a standard neighborhood of $\beth$ and recall that $O(\scrv,\eta)\subseteq (V_{x_0}\times I)\cup (X\times [0,1/3)\cup (2/3,1])$. The continuity of $\alpha$ at $0$ ensures that there exists a $N\in\bbn$ such that $\alpha([0,a_{1-N}])\subseteq q(O(\scrv,\eta))$. Thus $\im(\alpha_n)\subseteq O(\scrv,\eta)$ for $n\leq -N$. Fix $n\leq -N$. If $\im(\alpha_n)\subseteq U^{\down}$, then $\im(\alpha_n)$ lies in $(X\times [0,2/3))\cap O(\scrv,\eta)\subseteq O(\scrv,\eta)^{\down}$. Therefore, \[\im(G_n)=\im(q\circ H_{\alpha_n}^{\down})\subseteq \dh(\im(\alpha_n))\subseteq \dh(q(O(\scrv,\eta)^{\down}))\subseteq q(O(\scrv,\eta)).\]On the other hand, if $\im(\alpha_n)\nsubseteq U^{\down}$, then $\im(\alpha_n)$ lies in $(X\times (1/3,1])\cap O(\scrv,\eta)$ and we have \[\im(G_n)=\im(q\circ H_{\alpha_n}^{\up})\subseteq \uh(\im(\alpha_n))\subseteq \uh(q(O(\scrv,\eta)^{\up}))\subseteq q(O(\scrv,\eta)).\]Thus $G([0,a_{1-N}]\times \ui)\subseteq q(O(\scrv,\eta))$, proving that $G$ is continuous at the points of $\{0\}\times \ui$. The symmetric argument shows that $G$ is continuous at the points of $\{1\}\times\ui$. Note that this analysis ensures that if $\im(\alpha)\subseteq q(O(\scrv,\eta))$, then $\im(G)\subseteq q(O(\scrv,\eta))$.

Now $G$ gives a path-homotopy in $\Sigma X$ from $\alpha$ to $\gamma(t)=G(t,1)$ where $\gamma_n=\gamma|_{[a_n,a_{n+1}]}$. Recall from \Cref{pathexample} that since $\alpha_n$ has endpoints $(y_n,s_n)$ and $(y_{n+1},s_{n+1})$, $\gamma_n$ is a reparameterization of either $(\lambda_{y_{n}})|_{[0,s_n]}^{-}\cdot c_{\x}\cdot \lambda_{y_{n+1}}|_{[0,s_{n+1}]}$ or $(\lambda_{y_n})|_{[s_n,1]}\cdot c_{\x}\cdot (\lambda_{y_{n+1}}|_{[s_{n+1},1]})^{-}$. After deleting the middle constant subloops, this factorization of $\gamma$ may be reassociated to a $\bbz$-indexed concatenation $\delta\equiv\prod_{n\in\bbz}\zeta_n$ where each factor $\zeta_n$ is an inverse pair $(\lambda_{y_{n}})|_{[0,s_n]}\cdot(\lambda_{y_{n}})|_{[0,s_n]}^{-} $ or $(\lambda_{y_n})|_{[s_n,1]}^{-}\cdot(\lambda_{y_n})|_{[s_n,1]}$ or a non-trivial loop of the form $\lambda_{y_n}^{\epsilon_n}$ for $\epsilon_n\in\{\pm\}$. Set $M=\{m\in\bbz\mid \zeta_m\equiv\lambda_{y_m}^{\epsilon_m}\}$, viewed as a suborder of $\bbz$ (note that $M$ is order isomorphic to either $\{1,2,\dots,m\}$, $\bbn$, the reverse order of $\bbn$, or $\bbz$). By canceling the factors $\zeta_n$, $n\in\bbz\backslash M$ (those which are inverse pairs), we see that $\gamma$ either contracts in $\im(\gamma)$ or is path-homotopic within $\im(\gamma)$ to a concatenation of the form $\beta=\prod_{m\in M}\lambda_{y_{m}}^{\epsilon_m}$. 

Set $C=\{x_0\}\cup\{y_{m}\mid m\in M\}$. Certainly, we have $\im(\beta)=\Sigma C$. Since the homotopy $\gamma\simeq \beta$ lies in $\im(\gamma)$, the composed homotopy $\alpha\simeq \gamma$ still lies in $q(O(\scrv,\eta))$ whenever $\im(\alpha)$ does. To complete the proof, we check that $C\in\scrc$, i.e. that $C$ clusters at $x_0$. Let $V$ be a neighborhood of $x_0$ in $X$ and consider the standard neighborhood $W=V\times\ui\cup (X\times [0,1/3)\cup (2/3,1])$ of $\beth$. The continuity of $\beta$ ensures that $\lambda_{y_m}^{\epsilon_m}(1/2)=(y_{m},1/2)\in q(W)$ for all but finitely many $m\in M$. In particular, $y_{m}\in V$ for all but finitely many $m\in M$. Thus $C$ clusters at $x_0$.
\end{proof}

In the next theorem, we use the irreducible loop case to prove the same result for arbitrary loops.

\begin{theorem}\label{standardform}
For every based loop $\alpha:\ui\to \Sigma X$, there exists $C\in\scrc$ and a surjective loop $\beta:\ui\to \Sigma C$, which is path-homotopic to $\alpha$ in $\Sigma X$. Moreover, if $\im(\alpha)$ lies in standard neighborhood $q(O(\scrv,\eta))$, then this path-homotopy may be chosen to have image in $q( O(\scrv,\eta))$.
\end{theorem}

\begin{proof}
The theorem is clear for the constant loop at $\x$. Consider a non-constant loop $\alpha:\ui\to \Sigma X$ based at $\x$. Let $\scru$ be the set of connected components of $\ui\backslash \alpha^{-1}(\x)$ and note that if $J\in \scru$, then $\alpha|_{\ov{J}}$ is an irreducible loop. Moreover, the continuity of $\alpha$ ensures that if $O(\scrv,\eta)$ is a standard neighborhood, then $\alpha(\ov{J})\subseteq q(O(\scrv,\eta))$ for all but finitely many $J\in\scru$. 

Applying \Cref{simplelooplemma} to each $\alpha|_{\ov{J}}$, we obtain $C_J\in\scrc$ and a surjective loop $\beta_{J}:\ov{J}\to \Sigma C_{J}$, which is path-homotopic to $\alpha|_{\ov{J}}$ in $\Sigma X$. Let $G_{J}:\ov{J}\times \ui\to \Sigma X$ be such a homotopy (constructed as in the proof of \Cref{simplelooplemma}) from $\alpha|_{\ov{J}}$ to $\beta_J$ so that if $\im(\alpha|_{\ov{J}})$ lies in standard neighborhood $q( O(\scrv,\eta))$, then $\im(G_J)\subseteq q( O(\scrv,\eta))$.

Define $G:\ui^2\to \Sigma X$ so that $G(\alpha^{-1}(\x)\times \ui)=\x$ and so that for every $J\in\scru$, $G$ agrees with $G_{J}$ on $\ov{J}\times \ui$. Given a standard neighborhood $ O(\scrv,\eta)$, we have  $\alpha(\ov{J})\subseteq q(O(\scrv,\eta))$ for all but finitely many $J\in\scru$ and thus $\im(G_J)\subseteq q( O(\scrv,\eta))$ for all but finitely many $J\in\scru$. The continuity of $G$ is a direct consequence of this observation. Note that if $\im(\alpha)\subseteq q( O(\scrv,\eta))$, then our choice of $G_J$ ensures that $\im(G)\subseteq q( O(\scrv,\eta))$. Thus the second statement of the theorem is established.

Set $\beta(t)=G(t,1)$ and $C=\bigcup_{J\in\scru}C_J$. Since $\im(\beta_J)= \Sigma C_J$, it's clear that $\im(\beta)= \Sigma C$. It suffices to check that $C\in\scrc$, i.e. that $C$ clusters at $x_0$. Let $V$ be a neighborhood of $x_0$ in $X$ and consider the standard neighborhood $W=V\times\ui\cup (X\times [0,1/3)\cup (2/3,1])$. Then $\Sigma C_J=\im(\beta_J)\subseteq \im(G_J)\subseteq q(W)$ for all but finitely many $J\in\scru$. Thus $C_J\subseteq V$ for all but finitely many $J\in\scru$. Since each set $C_J$ clusters at $x_0$, it follows that $C$ clusters at $x_0$.
\end{proof}

\begin{remark}\label{structureofbeta}
We take a moment to highlight an important feature of the maps constructed in the proofs of \Cref{simplelooplemma} and \Cref{standardform}, which will be referred to later on. Starting with an arbitrary loop $\alpha$, the conclusion of \Cref{standardform} produces a set $C\in\scrc$, a loop $\beta$, and a homotopy $G$ from $\alpha$ to $\beta$. The constructions used ensure that for each component $(a,b)$ of $\ui\backslash \beta^{-1}(\x)$, we have $\beta|_{[a,b]}\equiv \lambda_{x}^{\epsilon}$ for some $x\in C$ and $\epsilon\in\{\pm\}$.
\end{remark}

\begin{remark}
Another interpretation of \Cref{standardform} is the following: for each $C\in\scrc$, let $f_C:\Sigma C\to \Sigma X$ be the inclusion map. Then \[\pi_1(\Sigma X,\x)=\bigcup_{C\in\scrc}(f_C)_{\#}(\pi_1(\Sigma C,\x)),\]that is, every element of $\pi_1(\Sigma X,\x)$ lies in the homomorphic image of a finitely generated free group or the infinite earring group where the homomorphism is induced by an inclusion map.
\end{remark}

Toward a characterization of the based spaces $(X,x_0)$ for which $\Sigma X$ is simply connected, we give the following definition.

\begin{definition}
A space $X$ is \textit{sequentially $0$-connected} at $x_0\in X$ if for every convergent sequence $\{x_k\}_{k\in\bbn}\to x_0$, there exists a sequence $\{\alpha_k\}_{k\in\bbn}$ of paths in $X$ that converges toward $x_0$ and such that $\alpha_k(0)=x_0$ and $\alpha_k(1)=x_k$ for all $k\in\bbn$.
\end{definition}

Certainly, if a space $X$ is first countable and locally path connected at $x_0$, then $X$ is sequentially $0$-connected at $x_0$.

\begin{lemma}
For any based space $(X,x_0)$, $\Sigma X$ is sequentially $0$-connected at $\x$.
\end{lemma}

\begin{proof}
Suppose $\{y_k\}_{k\in\bbn}\to\x$ in $\Sigma X$. We may assume that $y_k\neq \x$ and thus $y_k=(x_k,t_k)\in X\times\ui\backslash\beth$ for all $k$. If $0<t_k\leq 1/2$, let $\beta_k\equiv  \lambda_{x_k}|_{[0,t_k]}$ and if $1/2<t_k<1$, let $\beta_k\equiv  \lambda_{x_k}|_{[t_k,1]}^{-}$. Then $\{\beta_k\}_{k\in\bbn}$ is a sequence of paths in $\Sigma X$ from $\x$ to $y_k$. We check that $\{\beta_k\}_{k\in\bbn}$ converges toward $\x$. Let $O(\scrv,\eta)$ be a standard neighborhood of $\beth$ so that $q(O(\scrv,\eta))$ is a basic neighborhood of $\x$ in $\Sigma X$. Then there exists $K$ such that $(x_k,t_k)\in O(\scrv,\eta)$ for all $k\geq K$. Fix $k\geq K$. If $(x_k,t_k)\in V_{x_0}\times [0,1]$, then it is clear that $\im(\beta_k)\subseteq q(O(\scrv,\eta))$. Otherwise, $(x_k,t_k)\in V_x\times [0,\eta(x))\cup (1-\eta(x),1]$ for some $x\in X$. If $ (x_k,t_k)\in V_x\times [0,\eta(x))$, then $\im(\beta_k)=\dh(q(x_k,t_k))\subseteq q(O(\scrv,\eta))$ and if $ (x_k,t_k)\in V_x\times (1-\eta(x),1]$, then $\im(\beta_k)=\uh(q(x_k,t_k))\subseteq q(O(\scrv,\eta))$. Thus $\im(\beta_k)\subseteq q (O(\scrv,\eta))$ for all $k\geq K$. 
\end{proof}

We now show that if $X$ is sequentially $0$-connected at $x_0$, then $\Sigma X$ is simply connected in a strong sense ``at $x_0$."

\begin{definition}
A space $X$ is \textit{sequentially $1$-connected} at $x_0\in X$ if $X$ is sequentially $0$-connected and if every sequence of loops $f_n:(S^1,\ast)\to (X,x_0)$, $n\in\bbn$ that converges toward $x_0$, there exists a sequence of maps $g_n:D^2\to X$, $n\in\bbn$ on the closed unit disk that converges toward $x_0$ and such that $g_{n}|_{S^1}=f_n$ for every $n\in\bbn$.
\end{definition}

\begin{proposition}\label{seq-1-connected}
If $X$ is sequentially $0$-connected at $x_0$, then $\Sigma X$ is sequentially $1$-connected at $\x$. In particular, $\pi_1(\Sigma X,\x)$ is trivial.
\end{proposition}

\begin{proof}
Suppose that $\{\alpha_n\}_{n\in\bbn}$ is a sequence of based loops in $\Sigma X$ that converges toward $\x$. By \Cref{standardform}, for each $n\in\bbn$, we may assume there is some $C_n\in\scrc$ such that $\im(\alpha_n)=\Sigma C_n$. Moreover, by the second statement of \Cref{standardform}, we may assume that the collection $\{C_n\mid n \in\bbn\}$ clusters at $x_0$. Thus if we set $C=\bigcup_{n\in\bbn}C_n$, we have that $C\in\scrc$. Since $X$ is sequentially $0$-connected at $x_0$, for each $c\in C$, there exists a path $\gamma_c$ from $x_0$ to $c$ such that the collection $\{\gamma_c\mid c\in C\}$ clusters at $x_0$. Fix an enumeration $C_n\backslash\{x_0\}= \{c_{n,1}, c_{n, 2}, \dots \}$ of the (possibly finite) set $C_n\backslash\{x_0\}$. Using the paths $\gamma_c$, we may construct paths $\delta_n:\ui\to X$ such that $\delta_n(0)=x_0$, $\delta_n(1/k)=c_{n,k}$ for all values of $k$ used in the enumeration of $C_n\backslash\{x_0\}$, and such that the sequence $\{\delta_n\}_{n\in\bbn}$ converges toward $x_0$. The map $\theta_n:C_n\to \ui$, $\theta_n(c_{n,k})=1/k$, $\theta_n(x_0)=0$ is an embedding and gives factorization $\alpha_n=\Sigma\delta_n\circ \Sigma\theta_n\circ\alpha'_n$ through the contractible space $\Sigma I$, in which $\alpha_n':I\to \Sigma C_n$ is the map agreeing with $\alpha_n$ under the inclusion $i_n:\Sigma C_n\to \Sigma X$.
\[\xymatrix{
I \ar[d]_-{\alpha'_n} \ar[r]^-{\alpha_n} & \Sigma X \\ 
\Sigma C_n \ar@{^{(}->}[ur]_{i_n} \ar[r]_{\Sigma\theta_n} & \Sigma I \ar[u]_{\Sigma\delta_n}
}\]
Thus the loops $\alpha_n$, $n\in\bbn$ contract by a sequence of null-homotopies that converge toward $\x$.
\end{proof}

\section{Fundamental groups of reduced suspensions are locally free}

We remind the reader that, throughout this paper, the space $X$ is assumed to be Hausdorff. Recall that a \textit{Peano continuum} is a connected compact metric space that is locally connected. Toward a proof of \Cref{mainthm}, we give the following definition. 

\begin{definition}
For a space $X$ with basepoint $x_0$, let $\scrp(X,x_0)$ denote the set of all subsets $A\subseteq X$ containing $x_0$ and which satisfy:
\begin{enumerate}
\item $A$ has countably many path components
\item each path component of $A$ is a Peano continuum
\item the path components of $A$ cluster at $x_0$.
\end{enumerate}
If it will not cause confusion, we may write $\scrp$ instead of $\scrp(X,x_0)$.
\end{definition}

While $\scrp$ evidently becomes a partial order under subset inclusion ($A\leq B$ if $A\subseteq B$), it is not immediately clear if $\scrp$ is directed with respect to this relation. We recall and establish some general topology results, which will help us verify that $\scrp$ is directed as well as aid in the subsequent applications of $\scrp$. 

\begin{theorem}[Hahn-Mazurkiewicz theorem]
A space $X$ is a Peano continuum if and only if it is Hausdorff and there exists a continuous surjection $[0,1]\to X$.
\end{theorem}

The Hahn-Mazurkiewicz theorem implies that Peano continua are locally path connected, since quotients of locally path-connected spaces are locally path connected. Another direct consequence is the following lemma.

\begin{lemma}\label{fin-union}
Let $\{A_i\}_{i\leq n}$ be a finite collection of subsets of $X$, each of which is a Peano continuum. If $\bigcup_{i\leq n}A_i$ is path connected, then it is a Peano continuum.
\end{lemma}


We also require the following fact, which applies more generally.

\begin{lemma}\label{loc-connected}
Let $C_1, C_2, \dots, C_n$ be a finite collection of closed subsets of a space $X$. If $C_i$ is locally connected for each $i\leq n$, then so is the union $\bigcup_{i\leq n}C_i$.
\end{lemma}

\begin{proof}
It is straightforward to see that the disjoint union $\coprod_{i\leq n}C_i$ is locally connected and that $\coprod_{i\leq n}C_i\to \bigcup_{i\leq n}C_i$ is a closed continuous surjection, hence a quotient map. The lemma then follows from the fact that the quotient of a locally connected space is locally connected.
\end{proof}

\begin{remark}\label{weakly-loc}
A space $X$ is \textit{weakly locally connected at $x\in X$} if there exists a basis of connected neighborhoods at $x\in X$, though these neighborhoods are not required to be open. If a space $X$ is locally connected at a point $x$, then it is clearly also weakly locally connected at $x$, however the converse does not always hold. While these properties are distinct at individual points, a space $X$ is weakly locally connected at all points $x\in X$ if and only if it is locally connected at all points $x\in X$ \cite[Theorem 2.5]{degrootmcdowell}.
\end{remark}

\begin{lemma}\label{connected}
Let $\scra=\{A_i\}_{i\in\bbn}$ be a countably infinite collection of subsets of $X$ which cluster at $x_0$, where each $A_i$ is a Peano continuum. If $\bigcup_{i\in\bbn}A_i$ is path connected, then $\bigcup_{i\in\bbn}A_i\cup \{x_0\}$ is a Peano continuum.
\end{lemma}

\begin{proof}
Let $A = \bigcup_{i\in\bbn}A_i\cup \{x_0\}$. We begin by showing that $A$ is compact and metrizable. Let $B_i = A_i\cup\{x_0\}\subset X$ for each $i$. Each $B_i$ is compact and metrizable since $A_i$ is, and thus $\prod_{i\in\bbn}B_i$ is compact and metrizable. As a closed subset of $\prod_{i\in\bbn}B_i$, the shrinking wedge $\sw_{i\in\bbn} (B_i, x_0)$ is compact and metrizable as well. Finally, $A$ is a Hausdorff quotient of the compact metrizable space $\sw_{i\in\bbn} (B_i, x_0)$, hence compact and metrizable.

We have that $A$ is connected since every open set containing $x_0$ intersects $\bigcup_{i\in\bbn}A_i$, which is path connected by assumption. If $a\in A$ and $a\neq x_0$, then by taking disjoint neighborhoods $U$ and $V$ of $a$ and $x_0$ respectively and noting that $V$ contains all but finitely many $A_i\in\scra$, we see that $U$ intersects only finitely many $A_i\in\scra$. Let $B$ be the union of the finitely many $A_i$ that intersect $U$. Then $B$ is locally connected by \Cref{loc-connected} and thus has a basis of connected open neighborhoods at $a$. Since we may assume the neighborhoods are contained in $U$, we see that $A$ is locally connected at all points $a\neq x_0$.

We now show that $A$ is weakly locally connected at $x_0$, which by \Cref{weakly-loc} will finish showing that $A$ locally connected, and thus a Peano continuum. Let $U$ be an open neighborhood of $x_0$ in $A$. We partition $\scra$ into three sets. Let $\scrb = \{A_i\in\scra\mid A_i\subseteq U\}$, $\scrd = \{A_i\in \scra\mid x_0\in A_i, A_i\nsubseteq U\}$, and $\scre = \{A_i\in \scra\mid x_0\notin A_i, A_i\nsubseteq U\}$. Note that $\scrd$ and $\scre$ are finite sets. For each $D\in \scrd$, find a connected open neighborhood $U_D$ of $x_0$ in $D$ such that $U_D\subseteq U$. Let $N\subseteq U$ be the union 
$$N = \bigcup_{D\in\scrd}U_D\cup\bigcup \scrb\cup\{x_0\}.$$
Though $N$ need not be open, it is a neighborhood of $x_0$ because $C=\bigcup_{D\in\scrd}(D\setminus U_D)\cup \bigcup\scre$ is a closed subset of $A$ whose complement $V=A\backslash C$ is open in $A$ and satisfies $x_0\in V\subseteq N$.

We now show that $N$ must have only finitely many connected components, in which case the connected component $W$ of $N$ containing $x_0$ is open in $N$. Toward a contradiction, suppose that $N$ has infinitely many connected components. Each connected component of $N$ distinct from $W$ is a finite union of elements of $\scrb$. Since $\bigcup_{i\in\bbn}A_i$ is path connected, each such component must intersect an element of either $\scrd$ or $\scre$. Since both $\scrd$ and $\scre$ are finite sets, there exists $A_N\in \scrd\cup\scre$ such that the set $F= A_N\setminus\bigcup_{D\in\scrd} U_D$ intersects an infinite number of the components of $N$. However, $F$ is a closed set that does not contain $x_0$ yet has non-trivial intersection with every neighborhood of $x_0$; a contradiction.

Finally, since $W$ is open in $N$, we may write $W=N\cap W'$ for a set $W'$ that is open in $A$. Then $V\cap W'$ is open in $A$ and $x_0\in V\cap W'=(V\cap N)\cap W'=V\cap W\subseteq W$. Thus $W$ is a connected neighborhood (in $A$) of $x_0$ contained in $U$. We conclude that $A$ is weakly locally connected at $x_0$.
\end{proof}

\begin{proposition}\label{components-are-Peano}
Let $\scra=\{A_i\}_{i\in \mci}$ be a countable collection of subsets of $X$ which cluster at $x_0$, and where each $A_i$ is a Peano continuum. Each path component of $\bigcup_{i\in \mci} A_i\cup\{x_0\}$ is a Peano continuum.
\end{proposition}

\begin{proof}
If $\scra$ is finite, then the proposition is clear from \Cref{fin-union} so we assume $\scra$ is infinite. Let $A=\bigcup_{i\in \mci} A_i\cup\{x_0\}$ and let $C$ be a path component of $A$. If infinitely many elements of $\scra$ are subsets of $C$, then \Cref{connected} implies that $C$ also contains $x_0$. So if $x_0\notin C$, then $C$ is the union of only finitely many $A_i\in \scra$, in which case \Cref{fin-union} implies that $C$ is a Peano continuum.

So now assume that $C$ is the path component of $A$ containing $x_0$. Let $\scrb=\{A_i\in\scra\mid A_i\subseteq C\}$. Note that a path component $P$ of $\bigcup \scrb$ either contains $x_0$ or contains infinitely many $A_i$. In either case, $P\cup\{x_0\}$ is a Peano continuum (using \Cref{fin-union} or \Cref{connected}). Thus if $\{P_j\}_{j\in J}$ is the set of path components of $\bigcup\scrb$, then $\{P_j\cup\{x_0\}\}_{j\in J}$ is a collection of Peano continua that clusters at $x_0$ and whose union is $C$. Finally, applying \Cref{fin-union} (when $J$ is finite) or \Cref{connected} (when $J$ is infinite) to the collection $\{P_j\cup\{x_0\}\}_{j\in J}$, it follows that $C$ is a Peano continuum.
\end{proof}

Let $\leq$ be the partial order on $\scrp$ defined by subset inclusion.

\begin{corollary}\label{directed}
$(\scrp, \leq)$ is directed set.
\end{corollary}

\begin{proof}
We must show that every pair in $\scrp$ has an upper bound in $\scrp$. So let $A_1, A_2\in \scrp$ and set $A=A_1\cup A_2$. It is straightforward to see that $A$ has countably many path components, and that these path components cluster at $x_0$. Moreover, each path component of $A$ is a Peano continuum by \Cref{components-are-Peano}. Hence $A\in \scrp$.
\end{proof}

If $A\in\scrp$, then $A$ is closed in $X$ and, according to \Cref{closed}, $\Sigma A$ may be identified canonically as a subspace of $\Sigma X$. We implicitly make this identification in the following results.

\begin{lemma}\label{loc-free}
For each $A\in\scrp$, $\pi_1(\Sigma A, \x)$ is either free on finitely many generators or isomorphic to $\pi_1(\bbe_1, b_0)$.
\end{lemma}

\begin{proof}
Let $\{A_i\}_{i\in \mci}$ denote the path components of $A$, where the index set $\mci\subseteq \bbn\cup\{0\}$ contains $0$, and $A_0$ is the path component of $x_0$. Let $a_0=x_0$ and for each $i\in \mci$ with $i\neq 0$, choose a point $a_i\in A_i$ and let $C=\{a_i\}_{i\in \mci}$. Let $r:A\to C$ be the retraction which maps each path component $A_i$ to $a_i$. Note that, when $i\neq 0$, $A_i$ is open in $A$ and therefore, $r$ is continuous on $A_i$. Next, consider a point $x\in A_0$. A basic open neighborhood of $r(x)=a_0$ in $C$ has the form $U=C\backslash \{a_i\mid i\in F\}$ for some finite set $F\subseteq \mci\backslash \{0\}$. Now $V=\bigcup\{A_i\mid i\in \mci\backslash F\}$ is an open neighborhood of $x$ such that $r(V)=U$. Thus $r$ is continuous at $x$. 

Let $\Sigma r:\Sigma A\to \Sigma C$ be the retraction induced by $r$. If $i_\#:\pi_1(\Sigma C,\x)\to \pi_1(\Sigma A,\x)$ is the homomorphism induced by inclusion, then $i_\#$ is injective since it has left inverse $\Sigma r_\#$. We now show that $i_\#$ is surjective. Let $[\alpha]\in \pi_1(\Sigma A,\x)$. By \Cref{standardform}, there exists $B\in\scrc(A,x_0)$ and a based loop $\beta:\ui\to \Sigma B$ that is path-homotopic to $\alpha$ in $\Sigma A$. Moreover, recall from \Cref{structureofbeta} that $\beta$ is constructed so that if $\mcj$ is the set of connected components of $\ui\backslash\beta^{-1}(\x)$, then for each $J\in \mcj$, $\beta|_{\ov{J}}\equiv \lambda_{x_J}^{\epsilon_{J}}$ for some $x_J\in B$ and $\epsilon_{J}\in\{\pm\}$. The continuity of $\beta$ ensures that $\{x_J\mid J\in\mcj\}$ clusters at $x_0$. For each $J\in\mcj$, we let $i_J$ denote the element of $\mci$ such that $x_J\in A_{i_J}$. Observe that, since $\{x_J\mid J\in\mcj\}$ clusters at $x_0$, the set $\{J\in\mcj\mid i_J=n\}$ can only be infinite if $n=0$. Now, for each $J\in\mcj$, we choose a path $\gamma_J:\ui\to A_{i_J}$ from $x_J$ to $a_{i_J}$ as follows. If $i_J\neq 0$, then the path $\gamma_{J}$ may be chosen arbitrarily. Note that the (possibly finite) set $\{x_J\mid i_J=0\}$ clusters at $x_0$. Since $A_0$ is a Peano continuum, it is sequentially $0$-connected at $x_0$. Therefore, we may choose the set $\{\gamma_J\mid i_J=0\}$ so that it clusters at $x_0$. With all paths $\gamma_J$, $J\in\mcj$ chosen in this way, it follows that the set $\{\gamma_J\mid J\in\mcj\}$ clusters at $x_0$.

For each $J\in\mcj$, let $H_J:\ov{J}\times \ui\to \Sigma A$ be the map defined by $H_J(s,t)=\lambda_{\gamma_J(t)}^{\epsilon_J}(s)$. Now, let $H:\ui^2\to \Sigma A$ be defined so that $H(\beta^{-1}(\x)\times \ui)=\x$ and $H(s,t)=H_J(s,t)$ when $s\in \ov{J}$. Then $H(s,0)=\beta(s)$ and $H(s,1)$ has image in $\Sigma C$. The continuity of $H$ is guaranteed since every neighborhood of $x_0$ in $X$ contains all but finitely many of the sets $\im(\gamma_J)$. It follows that any neighborhood of $x_0$ in $\Sigma X$ contains all but finitely many of the sets $\im(H_J)$. Then the homotopy class of the loop $H(s,1)$ is mapped to $[\alpha]$ by $i_\#$, showing that $i_\#$ is surjective.

Therefore, $i_\#$ is an isomorphism. We finish by observing that $\pi_1(\Sigma C)$ is free on finitely many generators when $C$ is finite and $\pi_1(\Sigma C)\cong \pi_1(\bbe_1)$ when $C$ is infinite.
\end{proof}

Let $\{\pi_1(\Sigma A,\x)\}_{A\in\scrp}$ be the direct system of groups indexed by $\scrp$ where the bonding maps $\phi_{AB}:\pi_1(\Sigma A, \x)\to \pi_1(\Sigma B,\x)$ for $A\leq B$ are induced by the inclusion $A\to B$. We may then consider the direct limit $\varinjlim_{A\in\scrp} \pi_1(\Sigma A, \x)$. For each $B\in \scrp$, let $\phi_B$ denote the map $\phi_B:\pi_1(\Sigma B, \x)\to \varinjlim_{A\in\scrp}\pi_1(\Sigma A,\x)$. Since for each $B\in\scrp$, the inclusion $B\to X$ induces a homomorphism $\psi_B:\pi_1(\Sigma B, \x)\to \pi_1(\Sigma X, \x)$, we obtain an induced homomorphism 
$$\psi:\varinjlim_{A\in\scrp} \pi_1(\Sigma A, \x)\to \pi_1(\Sigma X, \x).$$

\begin{lemma}\label{inj}
Let $A_1\in\scrp$, let $\alpha:I\to \Sigma A_1$ be a loop based at $\x$, and suppose that $\alpha$ is null-homotopic in $\Sigma X$. Then there exists $A_2\in\scrp$ such that $\alpha$ is null-homotopic in $\Sigma A_2$.
\end{lemma}

\begin{proof}
Let $\alpha:I\to \Sigma A_1$ be a loop based at $\x$ and suppose $H:I^2\to \Sigma X$ is a null-homotopy of $\alpha$. Since any second countable surface (with boundary) admits a triangulation \cite[p. 107]{Ahlfors}, let $T:K\to I^2\setminus H^{-1}(\x)$ be a triangulation of $I^2\setminus H^{-1}(\x)$. To simplify the notation, we identify $K$ with $I^2\setminus H^{-1}(\x)$ under $T$ so that a simplex $\sigma$ of $K$ is equally regarded as the closed subset $T(\sigma)$ of $I^2\setminus H^{-1}(\x)$. Let $K_m$ denote the set of $m$-simplices in the simplicial complex $K$. Let $U_0, U_1$ be the open sets in $K$ defined by 
$$U_0=H^{-1}\Big(q\big((X\setminus x_0)\times (0, 2/3)\big)\Big),\quad U_1=H^{-1}\Big(q\big((X\setminus x_0)\times (1/3, 1)\big)\Big).$$
Since $U_0$ and $U_1$ cover $K$, by a theorem of J.H.C. Whitehead \cite[Theorem 35]{JHCWhitehead}, there exists a subdivision $K'$ of $K$ such that every simplex of $K'$ is contained in either $U_0$ or $U_1$. Thus, by replacing $K$ with $K'$, we may assume that every simplex of $K$ is contained in either $U_0$ or $U_1$. Let $\scrs_0$ be the collection of all $\sigma\in K_2$ such that $\sigma\subseteq U_0$, and let $\scrs_1=K_2\setminus \scrs_0$. Note that if $\sigma\in \scrs_1$, then $\sigma\subseteq U_1$. Let $\scrb$ denote the collection of all $\tau\in K_1$ such that there exists $\sigma\in \scrs_0$ and $\sigma'\in \scrs_1$ with $\tau=\sigma\cap\sigma'$, and let $B=\bigcup\scrb$. Since $H(B)\subseteq q((X\setminus x_0)\times (1/3, 2/3))$, let $B'$ be the image of $H(B)$ under the projection $p:q((X\setminus x_0)\times (1/3, 2/3))\to X\setminus x_0$. Finally, set $A_2=A_1\cup B'$. 

To show that $A_2\in\scrp$, it suffices to show that $B'\cup\{x_0\}\in \scrp$. First write $\scrb=\{\tau_i\}_{i\in \mci}$, and let $t_i =p(H(\tau_i))\subseteq X$ for each $i\in \mci$. Observe that $t_i$ is a Peano continuum for each $i\in \mci$ by the Hahn-Mazurkiewicz theorem. If $U$ is a neighborhood of $x_0$ in $X$, let 
$$V=q\Big((U\times[0,1])\cup (X\times[0, 1/3)\cup (2/3, 1])\Big).$$
Since $V$ is an open neighborhood of $\x$ in $\Sigma X$, $H^{-1}(V)$ is an open set in $I^2$ which contains $H^{-1}(\x)$. Then $C=I^2\setminus H^{-1}(V)$ is a compact subset of $I^2\setminus H^{-1}(\x)$. This implies that $C$ intersects $\tau_i$ for only finitely many $i\in \mci$ since $K$ is necessarily locally finite. Then $V$ contains $H(\tau_i)$ for all but finitely many $i\in \mci$. Due to the fact that $H(\tau_i)\subseteq q(X\times (1/3, 2/3))$ for all $i\in\mci$, it follows that $U$ contains $t_i$ for all but finitely many $i\in \mci$. Hence we may apply \Cref{components-are-Peano} to see that $B'\cup \{x_0\}=\bigcup_{i\in \mci}t_i\cup  \{x_0\}$ is in $\scrp$. Thus $A_2\in \scrp$.

We now construct a map $H':I^2\to \Sigma A_2$ such that $H'|_{\partial I^2}=H|_{\partial I^2}$, which will prove the lemma. Let $C_{0}$ be the union of all $\tau\in K_1$ such that there is exactly one $\sigma\in \scrs_0$ such that $\tau\subset \sigma$. Likewise, let $C_{1}$ be the union of all $\tau\in K_1$ such that there is exactly one $\sigma\in \scrs_1$ such that $\tau\subset \sigma$. Let $S_{0}=\bigcup \scrs_0$ and $S_{1}=\bigcup \scrs_1$. Then $(S_{0}, C_{0})$ and $(S_{1}, C_{1})$ are simplicial complex pairs. Let $f=H|_{S_{0}}$ and $g=H|_{S_{1}}$. Since $f$ and $g$ have image in $\Sigma X\setminus \x$, we may regard them as maps with codomain $X\times I$. Let $f_{C_{0}}^{\down}$ be the pushdown of $f$ rel. $C_0$ as described in \Cref{push-up-down}. Similarly, let $g_{C_{1}}^{\up}$ be the pushup of $g$ rel. $C_1$. We define a map $H':I^2\to \Sigma X$ by setting $H'|_{S_{0}}=q\circ f_{C_{0}}^{\down}$, $H'|_{S_{1}}=q\circ g_{C_{1}}^{\up}$, and $H'(H^{-1}(\x))=\x$. Note that $I\times\{0\}\subset C_0\cup C_1\cup H^{-1}(\x)$, hence $H'|_{I\times\{0\}}=H|_{I\times\{0\}}=\alpha$.

By the definition of $f^{\down}_{C_0}$ and $g^{\up}_{C_1}$, we have that $H'$ has image in $\dh(H(C_0))\cup \uh(H(C_1))$. Since $C_0\cup C_1\subseteq (I\times\{0\})\cup B$, we have that $H(C_0)\cup H(C_1)\subseteq \im(\alpha)\cup H(B)$. It follows that $H'$ has image in $\Sigma A_2=\Sigma (A_1\cup B')$. All that remains is to verify that $H'$ is continuous. Since $f$ and $g$ agree on $B=C_{0}\cap C_{1}=S_{0}\cap S_{1}$, so too do $f^{\down}_{C_{0}}$ and $g^{\up}_{C_1}$. Thus $H'$ is continuous at all points in $S_{0}\cup S_{1}$. Now suppose that $y\in I^2\setminus(S_{0}\cup S_{1})=H^{-1}(\x)$. Then $H'(y)=\x$, so let $U=q(O(\scrv, \eta))$ be a standard open neighborhood of $\x$ in $\Sigma X$, and let $U^\down = q(O(\scrv, \eta)^\down)$ and $U^\up = q(O(\scrv, \eta)^\up)$. Since $H^{-1}(U)$ is an open set in $I^2$ containing $H^{-1}(\x)$, we have that $H^{-1}(U)$ contains all but finitely many $\sigma\in K_2$. So let $F$ be the finite union of all $\sigma\in K_2$ such that $\sigma$ is not contained in $H^{-1}(U)$. Then $F$ is a closed set whose complement $V=I^2\setminus F$ is an open set in $I^2$ which contains $H^{-1}(\x)$ and satisfies $\ov V\subseteq H^{-1}(U)$. Recalling \Cref{updownsimplices}, we have that $f^{\down}_{C_0}(\sigma)\subset \dh(f(\sigma))$ for all $\sigma\in \scrs_0$, and $g^{\up}_{C_1}(\sigma)\subset \uh(g(\sigma))$ for all $\sigma\in \scrs_1$. Hence $H'(V)\subseteq \dh(H(\ov V\cap S_0))\cup \uh(H(\ov V\cap S_1))$. Since $H(\ov V\cap S_0)\subseteq q(X\times [0, 2/3))\cap U$, we have that $H(\ov V\cap S_0)\subseteq U^{\down}$. Consequently, $\dh(H(\ov V\cap S_0))\subseteq U^\down\subseteq U$. By a symmetric argument, $\uh(H(\ov V\cap S_1))\subseteq U^\up\subseteq U$. Thus, $V$ is an open neighborhood of $y$ such that $H'(V)\subset U$. Hence $H'$ is continuous.
\end{proof}

\begin{corollary}\label{adjoint-paths}
Two adjoint paths $\lambda_x$ and $\lambda_y$ are homotopic in $\Sigma X$ if and only if $x$ and $y$ lie in the same path component of $X$.
\end{corollary}

\begin{proof}
If $\alpha:I\to X$ is a path from $x$ to $y$, then $H:I^2\to \Sigma X$, $H(s,t)=q(\alpha(t),s)$ is a homotopy rel. $\partial I$ from $\lambda_x$ to $\lambda_y$. Conversely,  suppose that $\lambda_x$ and $\lambda_y$ are homotopic. Then $\lambda_x\cdot\lambda_{y}^{-}$ is nullhomotopic in $\Sigma X$. \Cref{inj} implies that there exists $A\in\scrp$ such that $\lambda_x\cdot\lambda_{y}^{-}$ is nullhomotopic in $\Sigma A$. Let $\{A_i\}_{i\in \mci}$ denote the path components of $A$, with $A_0$ being the path component of $x_0$. Let $a_0=x_0$ and for each $i\in\mci$, choose a point $a_i\in A_i$. Let $C=\{a_i\}_{i\in\mci}$ and let $\Sigma r:\Sigma A\to \Sigma C$ be the retraction as in the proof of \Cref{loc-free}. We have that $\Sigma r\circ\lambda_x=\lambda_{a_i}$ and $\Sigma r\circ\lambda_y=\lambda_{a_j}$ for some choice of $i,j\in\mci$. But then
$$1=[\Sigma r\circ(\lambda_x\cdot\lambda_y^-)]=[\lambda_{a_i}][\lambda_{a_j}]^{-1}$$
in $\pi_1(\Sigma C,\x)$. However, $\Sigma C$ is homeomorphic to a finite wedge of circles or $\bbe_1$ and in such a space, we must have $i=j$. Hence $x$ and $y$ are both contained in $A_i=A_j$, a path-connected subset of $X$.
\end{proof}

Finally, we prove our main result.

\begin{proof}[Proof of \Cref{mainthm}]
By \Cref{standardform}, every element of $\pi_1(\Sigma X,\x)$ is of the form $i_\#([\alpha])$ for some $[\alpha]\in \pi_1(\Sigma C, \x)$ where $C\subseteq X$ is a countable set which clusters at $x_0$ and $i:\Sigma C\to \Sigma X$ is the inclusion. Since $C\in \scrc\subseteq\scrp$, we have that $\psi$ is surjective. For injectivity, let $a\in \varinjlim_{A\in\scrp} \pi_1(\Sigma A, \x)$ and suppose that $\psi(a)=1$. There exists $B\in \scrp$ and $[\alpha]\in\pi_1(\Sigma B, \x)$ such that $\phi_B([\alpha])=a$. The fact that $\psi(a)=1$ and $\psi(a)=\psi\circ \phi_B([\alpha])=\psi_B([\alpha])$ implies that $\alpha$ is null-homotopic in $\Sigma X$. \Cref{inj} implies that there exists $C\in\scrp$ such that $\alpha$ is null-homotopic in $\Sigma C$, that is, $\phi_{BC}([\alpha])=1$. Hence $a=\phi_B([\alpha])=\phi_{C}\circ \phi_{BC}([\alpha])=1$. This shows that $\psi$ is injective.
\end{proof}

\begin{theorem}\label{equiv}
For a Hausdorff space $X$, the following are equivalent:
\begin{enumerate}
\item $X$ is sequentially 0-connected at $x_0$,
\item $\Sigma X$ is sequentially $1$-connected at $\x$,
\item $\Sigma X$ is simply connected.
\end{enumerate}
\end{theorem}

\begin{proof}
The implication $(1)\Rightarrow (2)$ is \Cref{seq-1-connected}, and $(2)\Rightarrow (3)$ follows by definition. We show that $(3)\Rightarrow (1)$ holds. So suppose that $\Sigma X$ is simply connected and $\{x_n\}_{n\in\bbn}\to x_0$ in $X$. We may assume that all terms in this sequence are distinct and not equal to $x_0$. Let $B=\{x_0\}\cup\{x_n\mid n\in\bbn\}$ and consider the $\bbn$-concatenation loop $\alpha=\prod_{n\in\bbn}\lambda_{x_n}$ in $\Sigma B$. By \Cref{mainthm} and the fact that $\Sigma X$ is simply connected, there exists $A\in\scrp$ with $B\subseteq A$ such that the inclusion $\iota: B\to  A$ gives $(\Sigma \iota)_{\#}([\alpha])=1$. Let $\scra=\{A_i\mid i\in\mci\}$ be the set of path components of $A$ where $A_0$ denotes the path component containing $x_0$. For each $i\in\mci$ pick a point $a_i\in A_i$, in particular, selecting $a_0=x_0$. Let $C=\{a_i\mid i\in \mci\}$ and note that $\Sigma C$ is homeomorphic to a point, a finite wedge of circles, or $\bbe_1$. Let $r:A\to C$ be the retraction mapping $A_i$ to $a_i$ and note that $x_n\in A_i$ if and only if $r(x_n)=a_i$. The induced homomorphism $(\Sigma r)_{\#}:\pi_1(\Sigma A,\x)\to\pi_1(\Sigma C,\x)$ gives $(\Sigma (r\circ \iota))_{\#}([\alpha])=1$. However, the loop $\Sigma (r\circ \iota)\circ \alpha= \prod_{n\in\bbn}\lambda_{r(x_n)}$ can only be null-homotopic in $\Sigma C$ if $r(x_n)=x_0$ for all $n\in\bbn$. Thus $x_n\in A_0$ for all $n\in\bbn$. Since $A_0$ is a Peano continuum, it is locally path connected and first countable, hence sequentially $0$-connected. Thus we may find a sequence of paths $\{\alpha_n\}_{n\in\bbn}$ in $A_0$ which clusters at $x_0$ and such that $\alpha_n(0)=x_0$ and $\alpha_n(1)=x_n$ for all $n\in\bbn$. We conclude that $X$ is sequentially $0$-connected at $x_0$.
\end{proof}

In \cite{Eda90}, K. Eda constructs a simply connected space $X$ which is locally simply connected at a point $x\in X$, yet not sequentially $1$-connected at $x$. The space constructed by Eda, as a set, is a quotient of an infinite number of copies of $C\bbe_1$, the unreduced cone over the infinite earring. The construction gives $X$ a topology which makes it not first countable at $x$. A similar construction, but instead using copies of the unreduced cone over the space $\bbe_0$, can be used to show that there exists a path connected space $X$ which is locally path connected at a point $x\in X$, yet not sequentially $0$-connected at $x$. Applying \Cref{equiv} to this space yields the following corollary.

\begin{corollary}
There exists a path-connected space $(X,x_0)$ which is locally path connected at its basepoint $x_0$ but such that $\Sigma X$ is not simply connected.
\end{corollary}

\end{document}